\documentclass[11pt, a4paper]{article}
\usepackage[utf8]{inputenc}
\usepackage{amsfonts, amssymb, amsmath, amsthm}
\usepackage{mathtools}
\usepackage{enumerate}
\usepackage{graphicx}
\usepackage{fullpage}
\usepackage{verbatim}
\usepackage{hyperref}
\usepackage[normalem]{ulem} % Crossing out an older text
\usepackage{authblk}
\usepackage{thmtools}
\hypersetup{colorlinks=true, linkcolor=red, urlcolor=blue, citecolor=blue, pdfpagemode=UseNone, pdfstartview=FitH, bookmarksopen=true}
\hypersetup{pdftitle=Title}

\newtheorem*{theorem*}{Theorem}
\newtheorem{theorem}{Theorem}

\newtheorem*{lemma*}{Lemma}
\newtheorem{lemma}[theorem]{Lemma}

\newtheorem*{proposition*}{Proposition}
\newtheorem{proposition}[theorem]{Proposition}
\newtheorem*{fact*}{Fact}

\newtheorem*{question*}{Question}
\newtheorem{question}[theorem]{Question}
\newtheorem{conjecture}[theorem]{Conjecture}
\newtheorem*{corollary*}{Corollary}
\newtheorem{corollary}[theorem]{Corollary}
\newtheorem*{claim*}{Claim}

\theoremstyle{remark}
\newtheorem*{remark*}{Remark}
\newtheorem{remark}[theorem]{Remark}
\newtheorem{example}[theorem]{Example}

\theoremstyle{definition}
\newtheorem*{definition*}{Definition}

\newtheorem*{observation*}{Observation}

\newcommand{\R}{\mathbb{R}}
\newcommand{\Z}{\mathbb{Z}}

\newcommand{\Q}{\mathbb{Q}}
\newcommand{\F}{\mathbb{F}}

\newcommand{\CC}{\mathcal{C}}
\newcommand{\HH}{\mathcal{H}}

\newcommand{\RR}{\mathcal{R}}
\renewcommand{\SS}{\mathcal{S}}
\newcommand{\UU}{\mathcal{U}}
\newcommand{\XX}{\mathcal{X}}

\DeclareMathOperator{\supp}{supp}

\DeclareMathOperator{\SPAN}{span}

\DeclareMathOperator{\cone}{cone}

\DeclareMathOperator{\rk}{rk}

\DeclareMathOperator{\val}{val}

\DeclareMathOperator{\evec}{\bf e}

\newcommand{\gvec}{{\bf g}}
\newcommand{\fvec}{{\bf f}}

\newcommand{\lip}{{\llcorner}}

\newif\ifcmnts
\cmntstrue %Comments on
\ifcmnts
\newcommand{\marrow}{\marginpar{\boldmath$\longleftarrow$}}
\newcommand{\denys}[1]{\ifhmode\newline\fi\marrow
  \textsf{\textcolor{magenta}{\bf
      Denys:} #1\newline}}
\newcommand{\martin}[1]{{\color{blue}\vskip 5pt\textsf{*** (Martin) \marrow #1\vskip 5pt}}}
\else
\newcommand{\denys}[1]{}
\newcommand{\martin}[1]{}
\fi

\title{Maximal matroids and counterexamples\thanks{D. B. is supported by AARMS postdoctoral fellowship. M. T. is supported by the
GA\v{C}R grant 25-16847S.}}
\author[1]{Denys Bulavka}
\author[2]{Martin Tancer}

\affil[1]{\small Department of Mathematics and Statistics, Dalhousie University, Halifax, Nova Scotia, Canada.}
\affil[2]{\small Department of Applied Mathematics, Faculty of Mathematics and Physics,
Charles~University, Prague, Czech Republic.}
\date{}

\begin{document}

\maketitle{}

\begin{abstract}
  Jackson and Tanigawa conjectured that the rigidity matroid $\RR^d_n$ and the hyperconnectivity matroid $\HH^d_n$ are the unique maximal
  matroids in the posets of $\{K_{d+2}, K_{d+2,d+2}\}$-matroids and $\{K_{d+2},K_{d+1,d+1}\}$-matroids, respectively. We disprove these conjectures
  by showing the existence of maximal matroids that are distinct from the proposed candidates.
\end{abstract}
  
\section{Introduction}

In a recent paper~\cite{jackson24-maximal}, Jackson and Tanigawa initiated a
systematic study of maximal matroids with prescribed sets of circuits.
They conjectured that certain well known matroids may be described as unique
maximal matroids among those with a particularly simple
prescribed set of circuits.

When trying to prove some of these conjectures, we actually arrived
at counterexamples, which is the content of this article.

\paragraph{An upper bound on the rank function of a matroid.}
Following~\cite{jackson24-maximal}, let $\XX$ be a family of subsets from a
finite set $E$. An \emph{$\XX$-matroid} on $E$ is a matroid whose ground set is $E$ and each
set in $\XX$ is a circuit. We consider matroids on $E$ ordered with respect to
inclusion. (Here and throughout the article we identify a matroid $M$ with the collection of its
independent sets.) Given $\XX$ as above, a \emph{proper $\XX$-sequence} is a sequence
$\SS = (X_1, \dots, X_k)$ of sets in $\XX$ such that, for every $i \in \{2,\dots, k\}$, $X_i$ is not a subset of
$X_1 \cup \cdots \cup X_{i-1}$. For $F
\subseteq E$, set $\val(F, \SS) := |F \cup X_1 \cup \cdots \cup X_k| - k$. Then
$\val_\XX \colon 2^E \to \Z$ is defined as 
\[
  \val_\XX(F) = \min \{ \val(F, \SS)\colon \SS \hbox{ is a proper
  $\XX$-sequence}\}.
\]

For any proper $\XX$-sequence $\SS$, the value $\val(F,\SS)$ is an upper bound on the rank of $F$ in every
$\XX$-matroid:

\begin{lemma}[{\cite[Part~of~Lemma 3.3]{clinch22-rigidity2}}]
\label{l:lower-bound}
  Let $M$ be an $\XX$-matroid on $E$ and $F\subseteq E$. Then $\rk_M(F)\leq
  \val(F,\SS)$ for any proper $\XX$-sequence $\SS$. In particular, this implies that
  $\rk_M(F) \leq \val_\XX(F)$.
\end{lemma}

Jackson and Tanigawa~\cite[Lemma~1.2]{jackson24-maximal} further showed that if
$\val_\XX$ is submodular and there exists some $\XX$-matroid, then there is a
unique maximal $\XX$-matroid with its rank function coinciding with $\val_\XX$.
The function $\val_\XX$ is interesting because it could describe the rank
function of various well known matroids as explained below.

\paragraph{Circuits in specific abstract rigidity matroids.} 

From now on let us assume that $E$ is the collection of edges of the complete graph
$K_n$, for some $n$. Given a collection of graphs $G_1, \dots, G_k$, with a
slight abuse of the notation, by setting $\XX = \{G_1, \dots, G_k\}$ we actually
mean that $\XX$ consists of all subgraphs of $K_n$ isomorphic to one of the
graphs $G_1, \dots, G_k$. 

Recently Clinch, Jackson and Tanigawa~\cite{clinch22-rigidity1,clinch22-rigidity2}
showed that the (generic) cofactor matroid\footnote{This matroid is denoted as
$\CC^1_2(K_n)$ in~\cite{clinch22-rigidity1}.} $\CC^3_n$ is the unique maximal $K_5$-matroid and
$\val_{K_5}$ is its rank function. Jackson and Tanigawa~\cite{jackson24-maximal} studied the existence of
maximal $\XX$-matroids for several classes $\XX = \{G_1, \dots,
G_k\}$ as above. They conjectured that various specific abstract rigidity
matroids~\cite{graver91-abstractrig} can be described as maximal $\XX$-matroids for a particularly simple family $\XX$
independent of $n$. In order to state these conjectures we recall that  
$\HH^d_n$ stands for the \emph{(generic) $d$-hyper\-connectivity matroid}\footnote{In Kalai's
paper this matroid is denoted $\HH_d^n$ but we prefer to be consistent
with~\cite{jackson24-maximal}.} on $n$ vertices introduced by
Kalai~\cite{kalai85-hyperconnectivity}, $\RR^d_n$ stands for the \emph{(generic) $d$-rigidity
matroid}\footnote{As far as we understand the notation
in~\cite{jackson24-maximal}, this matroid is denoted either $\RR_d(K_n)$ or
$\RR_n^d(K_n)$.} on $n$ vertices~\cite{maxwell64-rigidity,gluck74-rigidity,asimow78-rigidity,asimow79-rigidity},
and $\CC^d_n$ stands for the \emph{(generic) $d$-cofactor matroid}\footnote{This matroid is usually denoted
as $\CC^{d-2}_{d-1}(K_n)$~\cite{whiteley96-geometry,jackson24-maximal}.} on $n$ vertices ~\cite{whiteley88-cofactor,
  whiteley91-cofactor, alfeld93-cofactor, whiteley96-geometry}.  

\begin{conjecture}{\cite[Conjecture 6.6.a]{jackson24-maximal}}
\label{c:hyperconnectivity}
  Let $d\geq 1$ and $\XX=\{K_{d+2},K_{d+1,d+1}\}$. Then, $\HH_n^d$ is the
  unique maximal $\XX$-matroid and $\val_\XX$ is its rank function.
\end{conjecture}

We remark that $n$ is not really quantified in~\cite[Conjecture
6.6.a]{jackson24-maximal}; however, it is meaningful to assume $n \geq 2d + 2$,
because then $K_n$ contains a copy of $K_{d+1,d+1}$.
The conjecture is true for $d=1$ since in this case the hyperconnectivity matroid
$\HH_n^1$ coincides with the cycle matroid, which is also a count matroid.
In this case, $\HH_n^1$ is the unique maximal $\{K_3,K_{2,2}\}$-matroid, in particular, because it is the unique maximal
$K_3$-matroid~\cite[Theorem 4.3]{jackson24-maximal}, and $\val_{K_3}$ is its rank function.
Although the conjecture is open for $\HH_n^2$, Bernstein provided a combinatorial characterization of its independent sets~\cite{bernstein17-2hyper}.
Here we disprove Conjecture~\ref{c:hyperconnectivity} for $d\geq 3$:

\begin{theorem}
  \label{t:rank-hd}
  Let $d\geq 3$, $n\geq 2d+2$ and $\XX = \{K_{d+2}, K_{d+1,d+1}\}$. Then, $\HH_n^d$ is not the unique
  maximal $\XX$-matroid. In particular, $\val_\XX$ is not the rank function of $\HH_n^d$.
\end{theorem}

We remark that it is still possible that $\HH_n^d$ could be described as the
unique maximal $\XX$-matroid for some finite set $\XX$ independent of $n$. It
is also possible that there exists a unique maximal $\{K_{d+2},
K_{d+1,d+1}\}$-matroid. Theorem~\ref{t:rank-hd} only implies that if such a
matroid exists, it is not $\HH_n^d$.

\begin{conjecture}{\cite[Conjecture 6.3]{jackson24-maximal}}
  \label{c:rigidity}
  Let $d\geq 3$ and $\XX=\{K_{d+2},K_{d+2,d+2}\}$. Then, $\RR_n^d$ is 
  the unique maximal $\XX$-matroid and $\val_\XX$ is the rank function of $\RR_n^d$.
\end{conjecture}  

We disprove Conjecture~\ref{c:rigidity} for $d\geq 5$ and $n \geq 2d + 4$ (when $K_n$
contains a copy of $K_{d+2,d+2}$):

\begin{theorem}
  \label{t:rank-rd}
  Let $d\geq 5$, $n\geq 2d+4$ and $\XX=\{K_{d+2},K_{d+2,d+2}\}$. Then, $\RR^d_n$ is not the
  unique maximal $\XX$-matroid. In particular, $\val_\XX$ is not the rank function of $\RR^d_n$.
\end{theorem}

Again, it is still possible that $\RR_n^d$ could be described as the
unique maximal $\XX$-matroid for some finite set $\XX$ independent of $n$ as
well as it is possible that there exists a unique maximal $\{K_{d+2},
K_{d+2,d+2}\}$-matroid.

\paragraph{Outline.}
The remainder of the article is organized as follows. In Section~\ref{s:preliminaries} we review
some known matroidal constructions and the
necessary properties about the hyperconnectivity, rigidity and cofactor matroids.
In Section~\ref{s:xxmatroid} we review the properties of $\XX$-uniform matroids introduced in~\cite{jackson24-maximal}, and study their behavior when intersected with another matroid.
In Section~\ref{s:counter-umax} we present two counterexamples proving Theorem~\ref{t:rank-hd} and Theorem~\ref{t:rank-rd}.
\section{Properties of several abstract rigidity matroids}
\label{s:preliminaries}

\begin{table}
\begin{center}
  \begin{tabular}{cccc}
    & $\HH_n^d$ & $\RR_n^d$ & $\CC^d_n$\\
  $K_{d+2}$ & circuit & circuit & circuit\\
  $K_{d+1,d+1}$ & circuit & independent ($d\geq 2$) & independent ($d\geq 2$)\\
  $K_{d+2,d+2}$ & dependent & circuit ($d\geq 3$) & independent ($d\geq 4$) \\
  \end{tabular}
\end{center}
\caption{
  The statements in the table hold for $d\geq 1$, whenever $d$ is not specified.
  For $d=1$ all these matroids coincide with the cycle matroid~\cite{whiteley96-geometry}.
  The claims on $\HH_n^d$ are proved in~\cite{kalai85-hyperconnectivity}.
  The claims regarding $\RR_n^d$ for $d=2,3$ are proved in~\cite[Sections 2,3, and 9]{whiteley96-geometry},
  while for $d\geq 4$ are provided in~\cite[Examples 11.1.6, 11.2.4]{whiteley96-geometry}.
  The claims regarding $\CC_n^d$ for $d=2,3,4$ are given in~\cite[Sections 5, 10 and 11.3]{whiteley96-geometry},
  while for $d\geq 5$ are provided in Lemma~\ref{l:bipartite_plus_edge}.
}
  \label{t:circuits}
\end{table}

\paragraph{Hyperconnectivity matroid.}
Here we skip the technical definition of the hyperconnectivity matroid
from~\cite{kalai85-hyperconnectivity} thought we provide it in the appendix. We
only point out several important properties of this matroid, all of them proved
in~\cite{kalai85-hyperconnectivity}. Given two isomorphic graphs $G$ and $H$ on
at most $n$ vertices, $G$ and $H$ have the same rank in $\HH_n^d$,
i.e., $\HH_n^d$ is a \emph{symmetric matroid}~\cite{kalai90-symmetric}.
In particular, $G$ is a circuit in $\HH_n^d$ if and only if $H$ is a
circuit in $\HH_n^d$. In addition, the property of being a circuit does not depend
on $n$: (any isomorphic copy of) $G$ is a circuit in $\HH_n^d$ if any only if (any isomorphic copy of)
$G$ is a circuit in $\HH_m^d$ assuming that $m,n \geq |V(G)|$. It is known
that $K_{d+2}$ and $K_{d+1,d+1}$ are circuits in $\HH_n^d$ (for $n \geq 2d+2$);
see Table~\ref{t:circuits}. In particular, $\HH_n^d$ is a $\{K_{d+2},
K_{d+1,d+1}\}$-matroid.

Given a graph $G$, by $\cone(G)$ we denote the \emph{cone over $G$}
which is obtained from $G$ by adding a (new) vertex and connecting it with all
other vertices of $G$. For $k\geq 2$, by $\cone^k(G)$ we denote the
\emph{$k$-th cone} over $G$ defined iteratively as $\cone^{k-1}(\cone(G))$.
For ease of notation, we set $\cone^0(G)=G$.

For the proof of Theorem~\ref{t:rank-hd} we need to know that certain
cones are dependent in $\HH_n^d$. As a motivation for the questions at the end of
the paper, we actually need the slightly stronger conclusion that they are
circuits:

\begin{proposition}
  \label{p:cone-hbip}
The cone $\cone^k(K_{d+1-k,d+1-k})$ is dependent in $\HH_n^d$
  for every $d\geq 1$ and $k=0,\dots,d$ (where $n$ is at least the size of the
  vertex set of $\cone^k(K_{d+1-k,d+1-k})$). In fact, it is a circuit.
\end{proposition}

 The core of the proof of Theorem~\ref{t:rank-hd} is to build another $\{K_{d+2},
K_{d+1,d+1}\}$-matroid where $\cone(K_{d,d})$ is independent.
 
 Proposition~\ref{p:cone-hbip} can be immediately deduced from the following
 proposition and the fact that $K_{d+1-k,d+1-k}$ is a circuit in $\HH^{d-k}_n$. 
 
\begin{proposition}
  \label{p:coning_in_H} 
  \begin{enumerate}[(a)]
  \item
A graph $G$ is independent in $\HH^d_n$ if and only if $\cone(G)$ is
      independent in $\HH^{d+1}_n$ (provided $n\geq |V(G)|+1$).
  \item 
A graph $G$ is a circuit in $\HH^d_n$ if and only if $\cone(G)$ is a circuit in $\HH^{d+1}_n$ (provided $n\geq |V(G)|+1$).
 \end{enumerate}
\end{proposition}

Of course, in order to deduce Proposition~\ref{p:cone-hbip}, it is sufficient
to use Proposition~\ref{p:coning_in_H}(b), obtaining the stronger conclusion.
However, there is a reason why do we state the two claims separately. Both
these claims (a) and (b) appear in~\cite{kalai85-hyperconnectivity} as Theorem~5.1 and
Corollary~5.2, respectively. However, the completeness of the proofs may
possibly be questioned.  

First of all, Corollary~5.2 in~\cite{kalai85-hyperconnectivity} is given
without a proof. When supplying a na\"{\i}ve proof, there is a difficulty if
we assume that $G$ is a circuit in $\HH^d_n$ and we want to deduce that
$\cone(G)$ is a circuit in $\HH^{d+1}_n$. The difficulty arises when we remove some
edge containing the apex from $\cone(G)$, it does not seem to immediately
follow from Theorem~5.1 in~\cite{kalai85-hyperconnectivity}, that the resulting
graph is independent in $\HH^{d+1}_n$. We are thankful to D\'{a}niel
Garamv\"{o}lgyi~\cite{garamvolgyi26pc} for pointing out this problem to us. For completeness we supply the proof of Proposition~\ref{p:coning_in_H}(b) in the appendix.

Theorem~5.1 in~\cite{kalai85-hyperconnectivity} is given with a proof but one
could argue that the proof may need some adjustment. We experienced a
difficulty with the step that the proof claims generic acyclicity while the
basis $(g_i^*)$ used in the proof does not seem to be generic due to the
condition (2) in that proof. However, Theorem~5.1
in~\cite{kalai85-hyperconnectivity} (i.e. Proposition~\ref{p:coning_in_H}(a))
can also be deduced from the following two facts: 
\begin{itemize}
    \item A graph $G$ is dependent in
$\HH^d_n$ if and only if $\{d+1,d+2\}$ is an edge in the exterior shifting $\Delta_e(G)$ of $G$ (this
follows from the description of the hyperconnectivity matorid
in~\cite{guterman_nevo25} above Conjecture~1.1; though tracking references may
also require supplying a proof). 
\item Coning commutes with exterior shifting~\cite[Corollary~5.4]{nevo05-basic}. In
particular $\{d+1,d+2\} \in \Delta_e(G)$ if and only if $\{d+2, d+3\} \in
\Delta_e(\cone(G))$.
\end{itemize}

Alternatively, Proposition~\ref{p:coning_in_H}(a) follows from
Proposition~\ref{p:cone_hypercycle_hd} in the appendix, which we
use in the proof of Proposition~\ref{p:coning_in_H}(b).

\paragraph{Rigidity matroid.}
We also skip the technical definition of the rigidity matroid $\RR_n^d$. It
shares many properties with the hyperconnectivity matroid. Two isomorphic
graphs have again the same rank in $\RR_n^d$. Likewise, $G$ is a circuit in
$\RR_n^d$ if any only if (any isomorphic copy of) $G$ is a circuit in $\RR_m^d$ 
assuming that $m,n \geq |V(G)|$, see~\cite{asimow78-rigidity}. 

For $d\geq 3$, the well known circuits in $\RR_n^d$
are $K_{d+2}$ and $K_{d+2,d+2}$ (for $n \geq 2d+4$); see Table~\ref{t:circuits}. 
From this we get that $\RR_n^d$ is a $\{K_{d+2}, K_{d+2,d+2}\}$-matroid.
Similarly as for the hyperconnectivity matroid, the following corollary is a consequence of the fact
that coning transforms circuits of $\RR_n^d$ into circuit in $\RR_n^{d+1}$~\cite[Lemma 5.3]{daniel22-rigidity}.
%~\cite[Theorem 5]{whiteley83-conerigidity}.

\begin{corollary}
  \label{c:cone-rbip}
The cone $\cone^k(K_{d+2-k,d+2-k})$ is a circuit in $\RR_n^d$
for every $d\geq 3$ and $k=0,\dots,d-3$ (where $n$ is at least the size of the circuit).
\end{corollary}

Similarly as in the case for the hyperconnectivity matroid,
we only need the fact that $\cone(K_{d+1,d+1})$ is dependent in $\RR_n^d$.
The core of the proof of Theorem~\ref{t:rank-rd} is to build another $\{K_{d+2},
K_{d+2,d+2}\}$-matroid where $\cone(K_{d+1,d+1})$ is independent.

\paragraph{Cofactor matroid.}
We also skip the technical definition of the cofactor matroid $\CC^d_n$.
It shares many of the properties with the hyperconnectivity and rigidity matroids since
all of these are abstract rigidity matroids, e.g., the ranks of $\HH_n^d$, $\RR_n^d$
and $\CC_n^d$ are all equal to $dn - \binom{d+1}{2}$ and all have $K_{d+2}$ as a circuit.
It is also well behaved with respect to the cone operation:
\begin{lemma}[{\cite[Theorem 11.3.3]{whiteley96-geometry}}]
  \label{l:cofactor-cone}
  Let $G$ be a graph, then $G$ is independent in $\CC^d_n$ if and only if
  $cone(G)$ is independent in $\CC^{d+1}_n$ (provided that $n \geq |V(G)|+1$).
\end{lemma}

The matroids $\HH^d_n, \RR^d_n$ and $\CC^d_n$ in general differ if we look at their
circuits.
For example, it is known that the complete bipartite graph $K_{d+1,\binom{d+1}{2}}$ is dependent in
$\HH_n^d$ when $d\geq 2$, see Table~\ref{t:circuits}. However, this graph is independent in
$\CC_n^d$~\cite[Example 11.3.12]{whiteley96-geometry} for $d\geq 1$.

\begin{lemma}
\label{l:bipartite_plus_edge}
  \begin{enumerate}
    \item Let $d \geq 3, n \geq 2d+2$ and let $H$ be a subgraph of $K_n$
      obtained from $K_{d+1,d+1}$ by adding an edge (of $K_n$) to it. Then $H$ is independent in $\CC^d_n$.
    \item Let $d \geq 4, n \geq 2d+4$ and let $H$ be a subgraph of $K_n$
      obtained from $K_{d+2,d+2}$ by adding an edge (of $K_n$) to it. Then $H$ is independent in $\CC^d_n$.
  \end{enumerate}
\end{lemma}
We will need the first item of the lemma in the proof of
Theorem~\ref{t:rank-hd} in order to build a suitable $\{K_{d+2},K_{d+1,d+1}\}$-matroid.
Similarly, the second item will be used in the proof of Theorem~\ref{t:rank-rd}
in order to build a suitable $\{K_{d+2},K_{d+2,d+2}\}$-matroid.

For the proof of Lemma~\ref{l:bipartite_plus_edge} we need the definition of a $(k,d)$-extension of a graph.
For it, let us consider a pair of integers $0\leq k \leq d$, a graph $G$, a $k$-matching $\evec = \{e_1,\dots, e_k\}$ in $G$, i.e., a matching of size $k$,
and a set of vertices $V =\{v_1,\dots, v_{d-k}\}$ in $G$ disjoint from all
the edges in $\evec$.
The \emph{$(k,d)$-extension of $G$ (with respect to $\evec$ and $V$)} is a new graph $G'$ obtained from $G$
by adding to it a new vertex $u$, connecting $u$ to the endpoints of the edges appearing in the matching $\evec$ as well as
to the vertices in $V$, and finally removing from it the edges appearing in $\evec$.
It is known that the independent sets in $\CC_n^d$ are preserved after performing
$(0,d),(1,d)$ and $(2,d)$-extensions~\cite[Lemma 11.3.6, Theorem 11.3.10,
  Theorem 11.4.1]{whiteley96-geometry}.

\begin{proof}
 In both cases, let us assume that $H$ is obtained from the complete bipartite graph with
  parts $\{a_1, \dots, a_k\}$ and $\{b_1, \dots b_k\}$ where $k \in
  \{d+1,d+2\}$ by adding an edge $e$ to it. If $e$ meets the complete bipartite graph,
  without loss of generality, we can assume that it meets it in $a_1$ (and
  possibly $a_2$). 

  We start with the complete graph on $\{a_1,\dots,a_d, b_1,b_2\}$ minus the
  edge $\{a_{d-1},a_{d}\}$. This graph, call it $G_1$, is independent as the complete graph
  $K_{d+2}$ is a circuit in $\CC^d_n$

  We perform $(0,d)$-extensions in $G_1$ at vertices $b_3, \dots, b_d$
  connecting them to $a_1, \dots, a_d$, obtaining a graph $G_2$, independent in
  $\CC^d_n$. Next we perform one more $(0,d)$-extension at $a_{d+1}$,
  connecting it to $b_1,\dots, b_d$, obtaining a graph $G_3$, still independent
  in $\CC^d_n$. 

  Now we perform a $(1,d)$-extension in $G_3$ at the vertex $b_{d+1}$, connecting it
  to $a_1, \dots, a_{d+1}$ at the cost of removing the edge $a_1a_3$ from
  $G_3$. This way we obtain a graph $G_4$ independent in $\CC^d_n$.

 If we are in the case of the first item of the lemma, we are essentially done.
  If $e$ is the edge $a_1a_2$, then $H$ is a subgraph of $G_4$. If $e$ is not
  the edge $a_1a_2$, then we can perform further one or two $(0,d)$-extensions
  in vertices of $H$ and we obtain a supergraph of $H$ independent in
  $\CC^d_n$.

 It remains to consider the second item of the lemma, assuming $d \geq 4$. We
  perform one more $(1,d)$-extension at the vertex $a_{d+2}$, connecting it to
  $b_1, \dots, b_{d+1}$ while removing the edge $b_1b_2$, obtaining a graph
  $G_5$. Next, we perform a $(2,d)$-extension at the vertex $b_{d+2}$
  connecting it to $a_1, \dots, a_{d+2}$ while removing the edges $a_1a_4$ and
  $a_2a_3$, obtaining $G_6$ independent in $\CC^d_n$.

  Now we finish as in the previous case: If $e$ is the edge $a_1a_2$, then $H$
  is a subgraph of $G_6$. If $e$ is not
  the edge $a_1a_2$, then we can perform further one or two $(0,d)$-extensions
  in vertices of $H$ and we obtain a supergraph of $H$ independent in
  $\CC^d_n$.
\end{proof}

\begin{remark}
  As the reader might have noticed, if $d$ is large, the subgraph of $G_5$ (or
  $G_6$) induced by vertices $a_1, \dots, a_d$ still contains many edges (and
  pairs of disjoint edges). Let $G$ be $K_{d+2, d+ \binom d2/2}$ if $\binom d2$ is
  even and $G$ be $K_{d+2, d+ \lceil\binom d2/2\rceil}$ if $\binom d2$ is odd. 
  By a suitable strategy of applying
  $(2,d)$-extensions, starting with $G_5$, possibly followed by a single
  $(1,d)$-extension, it is possible to achieve that the graph $G$ is
  independent in $\CC^d_n$ for $d \geq 4$. This generalizes the well known fact that $K_{6,7}$ is
independent in $\CC_n^4$~\cite[Theorem 11.4.2]{whiteley96-geometry}.
  
  If, in addition, $n = 2d+2+
  \lceil\binom d2/2\rceil$, then $G$ is a basis (by comparing the number of
  edges of $G$ with the rank of $\CC^d_n$).
\end{remark}

\section{Intersecting a matroid with an $\XX$-uniform matroid}
\label{s:xxmatroid}

\paragraph{Matroid intersection.}
Let us denote by $\CC(M)$ the circuits of a matroid $M$.
We observe that if an intersection $N=M_1\cap M_2$ is a matroid,
then we can express its circuits and the rank function in terms
of $M_1$ and $M_2$ as follows.

\begin{lemma}
  \label{l:intersection-crank}
Let $M_1,M_2$ be matroids on the same ground set such that $N=M_1\cap M_2$ is a matroid as well. Then, $$\CC(N)= (\CC(M_1)\cap M_2) \cup (M_1\cap \CC(M_2))\cup (\CC(M_1)\cap \CC(M_2))$$ and $$r_N(X) = \min \{r_1(Y)+r_2(X\setminus Y)\colon Y\subseteq X\}.$$
\end{lemma}
We remark that we do not really need a formula for the rank function in
the lemma above. We add it for completeness.
\begin{proof}
  Let $C\in \CC(N)$. Then, since $C\setminus x\in N$ for every $x\in C$, we
  have that $C$ is either a circuit or independent in $M_1$ and $M_2$.
  Vice versa, if $C$ is either a circuit or independent in $M_1$ and
  $M_2$, and, in addition, $C \not\in N$, then $C \in \CC(N)$.
  The given decomposition covers all the possible combinations.
  Now, let $X\subseteq E$ then its rank in $N$ is given by $$r_N(X) = \max \{|I| \colon I\in M_1|_X \cap M_2|_X\} = \min\{r_1(Y)+r_2(X\setminus Y) \colon Y\subseteq X\}$$ where the second equality follow from~\cite[Corollary 11.3.16]{oxley11-matroid}.
\end{proof}

\begin{corollary}
  Fix a ground set $E$ and $\XX_1,\XX_2$ finite families of subsets of $E$.
  Let $M_1$ be an $\XX_1$-matroid, and $M_2$ an $\XX_2$-matroid such that $N=M_1\cap M_2$ is a matroid.
  Then, $\XX_i\subseteq M_j\cup \CC(M_j)$ for $i\neq j$ if and only if $N$ is an $(\XX_1\cup \XX_2)$-matroid.

  Moreover, if for $i=1,2$ the matroid $M_i$ is the unique maximal $\XX_i$-matroid for $i=1,2$,
  then $N$ is the unique maximal $(\XX_1\cup \XX_2)$-matroid.
\end{corollary}

\begin{proof}
  The matroid $N$ is an $(\XX_1\cup \XX_2)$-matroid if and only if $\XX_1\cup \XX_2 \subseteq \CC(N)$.
  By Lemma~\ref{l:intersection-crank} this last condition holds if and only if
  $$\XX_i\subseteq (\CC(M_1)\cap M_2) \cup (M_1\cap \CC(M_2))\cup (\CC(M_1)\cap \CC(M_2)).$$
  Since $\XX_i\cap M_i =\emptyset$ and $\XX_i \subseteq \CC(M_i)$, we have that $\XX_1\cup \XX_2 \subseteq \CC(N)$ holds if and only if
  $\XX_i \subseteq M_j \cup \CC(M_j)$, as required.
 
  Now, assume that $M_i$ is the unique maximal $\XX_i$ matroid for $i =
  1,2$. Let $N'$ be an $(\XX_1\cup \XX_2)$-matroid. In particular, $N'$ is a $\XX_i$-matroid
  and consequently, by uniqueness and maximality of $M_i$, we have that $N'\subseteq M_i$ for $i=1,2$.
  From this it follows that $N'\subseteq M_1\cap M_2 = N$.
\end{proof} 

The intersection of two matroids as defined in this section is not always a matroid.
Nash-Williams introduced the following closely related construction.
For two matroids $M_1,M_2$ on $E$ the family $M_1\vee M_2 = \{I_1\cup I_2 \colon I_1\in M_1, I_2\in M_2\}$
is again a matroid, see~\cite[Theorem 11.3.1]{oxley11-matroid}.
By matroid duality we obtain that $M_1\wedge M_2 = (M_1^*\vee M_2^*)^*$ is a matroid as well.
We observe that the inclusion $M_1\wedge M_2 \subseteq M_1\cap M_2$ holds in general.
Indeed, let $B$ be a basis of $M_1\wedge M_2$, then $B=E\setminus B^*$ with $B^*$ a basis of $M_1^*\vee M_2^*$.
Since $B^*$ is independent in $M_1^*\vee M_2^*$ then $B^*=I_1\cup I_2$ for $I_j\in M_j^*$.
For $j=1,2$, let $I_j\subseteq B_j^*$ be a basis in $M_j^*$. Then, $B^*\subseteq B^*_1\cup B^*_2 \subseteq B^*$ where the last inclusion
follows from the fact that $B^*$ is a maximal independent set in $M_1^*\vee M_2^*$.
Since $B_j^* = E\setminus B_j$ with $B_j$ a basis in $M_j$ we have that $B = B_1\cap B_2$.
However, the reverse inclusion $M_1\wedge M_2 \supseteq M_1\cap M_2$ does not need to be satisfied in general.
For it, consider $\UU_2$ the uniform matroid of rank $2$. We have that $\UU_2\cap \UU_2 = \UU_2$, while $\UU_2\wedge \UU_2 = \{\emptyset\}$ is the matroid whose only independent set is the empty set.

\paragraph{Intersection of $\XX$-uniform matroids.}
Jackson and Tanigawa~\cite{jackson24-maximal} introduced the notion of a $\XX$-uniform matroid.
For it, let $\XX$ be a finite uniform family of (finite) sets, i.e., all of its  members have the same size $k\geq 1$,
and set $E = \bigcup_{X \in \XX} X$.
We define the $\XX$-uniform system as
\[
  \UU_\XX = \{ F \subseteq E \colon |F| \leq k \hbox{ and there is no } X \in
  \XX \hbox{ with } X \subseteq F\}.
\]

We say that the family $\XX$ as above is \emph{union-stable} if for every $X_1,
X_2 \in \XX$ with $X_1 \neq X_2$ and $e \in X_1 \cap X_2$ either $|(X_1 \cup X_2) - e| > k$ or
there exist $X\in \XX$ such that $X \subseteq (X_1\cup X_2) - e$.\footnote{Thinking of $\XX$ as a set of circuits, these
conditions are essentially the circuit exchange axiom.
}

\begin{lemma}[{\cite[Lemma 2.5]{jackson24-maximal}}]
\label{l:union-stable}
 Let $\XX$ be a uniform family of sets. Then, $\XX$ is union-stable if and only if $\UU_\XX$ is a matroid.
\end{lemma}

\begin{example}
  \label{ex:bipartite}
  Let $d\geq 2$, $n\geq 2d$, and $\XX=\{K_{d,d}\}$ be given by all the isomorphic copies of $K_{d,d}$ in $K_n$.
  We claim that $\UU_{K_{d,d}} \coloneqq \UU_\XX$ is a matroid. For it, we verify that the family $\XX$ is union-stable.
  Indeed, any two distinct element $X_1,X_2\in \XX$ either share all the
  vertices or not. In the former case $|X_2 \setminus X_1| \geq 2$, in the latter
  case $|X_2 \setminus X_1| \geq d$.
  In both cases $|(X_1\cup X_2) \setminus e| > d^2$ for every $e\in X_1\cap X_2$.
\end{example}

Let $M$ be a matroid with the ground set $E$, a family $\XX$ of subsets of $E$
is \emph{compatible with} $M$ if for every $X\in \XX$, $e\in X$ and $f\in E$,
the set $(X\setminus e)\cup f \in M$.

\begin{lemma}
  \label{l:intersection}
  Let $M$ be a matroid with ground set $E$, and $\XX$
  a uniform family of subsets of $E$, all of whose elements have the same size $k\geq 1$, such that $\UU_\XX$ is a matroid.
  If $\XX$ is
  compatible with $M$, then $N=M\cap \UU_\XX$ is a matroid.
\end{lemma}

\begin{proof}
  We proceed by verifying the independent set axioms for $N$.
  The family of sets $N$ is not empty since both $M$ and $\UU_\XX$ contain the emptyset.
  Let $I_1,I_2\in N$ with $|I_1|< |I_2|$.
  Since $M$ is a matroid there exists $e\in I_2\setminus I_1$ such that
  $I_1\cup e\in M$. If $I_1\cup e\in \UU_\XX$ we are done.
  Assume that this is not the case, then by definition of $\UU_\XX$ either $|I_1\cup e| > k$
  or there exists $X\in \XX$ such that $X\subseteq I_1\cup e$.
  The former option is not possible since $|I_1\cup e| \leq |I_2| \leq k$.
  
  In the latter case we obtain that $|X|\leq |I_1\cup e| \leq |I_2|\leq |X|$ and consequently $X=I_1\cup e$,
  i.e., $I_1 = X\setminus e$.
  Because $\UU_\XX$ is a matroid, there exists $f\in I_2\setminus I_1$
  such that $I_1\cup f \in \UU_\XX$. We notice that $f\neq e$ since we assumed
  that $I_1 \cup e \notin \UU_\XX$. 
  Since $\XX$ is compatible with $M$ we can conclude that $(X\setminus e) \cup f = I_1\cup f$ is independent in $M$,
  and so it is independent in $N$ as well.
\end{proof}

\begin{example}
  \label{ex:bip-comp-small}
  For $d\geq 3$ and $n\geq 2d+2$, let $\XX = \{K_{d+1,d+1}\}$ denote the family of all the
  isomorphic copies of $K_{d+1,d+1}$ in $K_n$. We claim that $\CC_n^d \cap \UU_{K_{d+1,d+1}}$ is
  a matroid. From Example~\ref{ex:bipartite} we already know that $\UU_{K_{d+1,d+1}}$ is a matroid.
  By Lemma~\ref{l:bipartite_plus_edge}, $\XX$ is compatible with
  $\CC_n^d$. Consequently, Lemma~\ref{l:intersection} implies that $\CC_n^d \cap \UU_{K_{d+1,d+1}}$ is
  a matroid.
\end{example}

\begin{example}
  \label{ex:bip-comp}
  For $d\geq 4$ and $n\geq 2d+4$, let $\XX = \{K_{d+2,d+2}\}$ denote the family of all the
  isomorphic copies of $K_{d+2,d+2}$ in $K_n$. We claim that $\CC_n^d\cap \UU_{K_{d+2,d+2}}$ is a matroid.
  Similarly as above, it follow from Example~\ref{ex:bipartite} that $\UU_{K_{d+2,d+2}}$ is a matroid. 
  By Lemma~\ref{l:bipartite_plus_edge}, $\XX$ is compatible with
  $\CC_n^d$. Consequently, Lemma~\ref{l:intersection} implies that $\CC_n^d \cap \UU_{K_{d+2,d+2}}$ is
  a matroid.
\end{example}

\section{Counterexamples to the unique maximal candidates}
\label{s:counter-umax}

\begin{proof}[Proof of Theorem~\ref{t:rank-hd}]
  Let $d\geq 3$ and $\XX=\{K_{d+2},K_{d+1,d+1}\}$.  
  We show that $\HH_n^d$ is not the unique $\XX$-maximal matroid.
  For it, let $N = \CC_n^d\cap \UU_{K_{d+1,d+1}}$. It follows from Example~\ref{ex:bip-comp-small} that $N$ is a matroid.\\
  {\bf Claim: $N$ is an $\XX$-matroid.}
  By Lemma~\ref{l:intersection-crank} it follows that $N$ is an $\XX$-matroid since $K_{d+2}\in \CC(\CC^d_n)\cap \UU_{K_{d+1,d+1}}$ and
  $K_{d+1,d+1}\in \CC_n^d\cap \CC(\UU_{K_{d+1,d+1}})$.\\
  {\bf Claim: there is a maximal $\XX$-matroid containing $N$.}
  The poset of all $\XX$-matroids containing $N$ is not empty since
  $N$ is in it. Set $N'$ to be a maximal element in this poset.\\
  {\bf Claim: $N'$ is different from $\HH_n^d$.}
  Now, for $d\geq 3$ the graph $K_{d,d}$ is independent in $\CC_n^{d-1}$~\cite[Example 11.3.12]{whiteley96-geometry}. 
  By Lemma~\ref{l:cofactor-cone} we can conclude that $G = \cone(K_{d,d})$ is
  independent in $\CC_n^d$.
  By cardinality, $G$ is independent in $\UU_{K_{d+1,d+1}}$, and so it is independent in $N \subseteq N'$.
  However, by 
  Proposition~\ref{p:cone-hbip}, the graph $G$ is dependent in $\HH_n^d$ and
  consequently $N'\neq \HH_n^d$. It follows that $\HH_n^d$ is not
  the unique maximal $\XX$-matroid.

  Finally, by~\cite[Lemma~1.2]{jackson24-maximal}, the function $\val_\XX$ is not the rank function $\HH_n^d$ since
  otherwise $\HH_n^d$ would be the unique maximal $\XX$-matroid.
\end{proof}

\begin{proof}[Proof of Theorem~\ref{t:rank-rd}]
  Let $d\geq 5$ and $\XX = \{K_{d+2},K_{d+2,d+2}\}$.
  We show that $\RR_n^d$ is not the unique $\XX$-maximal matroid.
  For it, let $N = \CC_n^d\cap \UU_{K_{d+2,d+2}}$. It follows from Example~\ref{ex:bip-comp} that $N$ is a matroid.\\
  {\bf Claim: $N$ is an $\XX$-matroid.}
  By Lemma~\ref{l:intersection-crank} it follows that $N$ is an $\XX$-matroid
  since $K_{d+2}\in \CC(\CC_n^d)\cap \UU_{K_{d+2,d+2}}$ and, by Lemma~\ref{l:bipartite_plus_edge}, $K_{d+2,d+2}\in \CC_n^d\cap \CC(\UU_{K_{d+2,d+2}})$.\\
  {\bf Claim: there is a maximal $\XX$-matroid containing $N$.}
  The poset of all $\XX$-matroids containing $N$ is not empty since
  $N$ is in it. Set $N'$ to be a maximal element in this poset.\\
  {\bf Claim: $N'$ is different from $\RR_n^d$.}
  Now, for $d\geq 5$ the graph $K_{d+1,d+1}$ is independent in $\CC_n^{d-1}$, see Table~\ref{t:circuits}.
  By Lemma~\ref{l:cofactor-cone} it follows that $G=\cone(K_{d+1,d+1})$ is independent in $\CC_n^d$.
  By cardinality, $G$ is independent in $\UU_{K_{d+2,d+2}}$, and so it is independent in $N \subseteq N'$.
  However, by Corollary~\ref{c:cone-rbip} the graph $G$ is dependent in $\RR_n^d$ and consequently $N'\neq \RR_n^d$.
  It follows that $\RR_n^d$ is not the unique maximal $\XX$-matroid.

  Finally, by~\cite[Lemma~1.2]{jackson24-maximal}, the function $\val_\XX$ is not the rank function $\RR_n^d$ since
  otherwise $\RR_n^d$ would be the unique maximal $\XX$-matroid.
\end{proof}

\section{Questions}
  
\paragraph{Circuits in the hyperconnectivity matroid.}
We disproved Conjecture~\ref{c:hyperconnectivity} due to the fact
that $\cone(K_{d,d})$ is dependent in $\HH_n^d$ while we found a $\{K_{d+2},
K_{d+1,d+1}\}$-matroid where $\cone(K_{d,d})$ is independent. This suggests that
the list of prescribed circuits in $\XX$, needed to get $\HH_n^d$ as a unique
maximal $\XX$-matroid, is incomplete. Maybe, it is possible to fix
Conjecture~\ref{c:hyperconnectivity} by adding circuits from Proposition~\ref{p:cone-hbip}. 

\begin{question}
  Let $d\geq 1$ and $\XX=\{\cone^k(K_{d+1-k,d+1-k})\colon k
  \in\{0,\dots,d\}\}$. Is it true that $\HH_n^d$ is the
unique maximal $\XX$-matroid and $\val_\XX$ is its rank function?
\end{question}

We remark that $\XX$ above contains, in particular, $K_{d+2}$ and $K_{d+1,d+1}$
by setting $k$ to $d$ and $0$ respectively. We also remark that, in principle,
it might be even possible to remove some elements from $\XX$ above in order to
answer the question affirmatively. For example, for $d=1$, we get $\XX = \{K_3,
K_{2,2}\}$. However, $\HH_n^1$ is already a unique maximal $K_3$-matroid (as we
mention in the introduction). In general, it may, for example, turn out that it
is not necessary to include $\cone^{d-1}(K_{2,2})$. However, we include all cones in $\XX$ for a convenient formulation. 

\paragraph{Rank of $\{K_{d+2},K_{d+1,d+1}\}$-matroids.}
In the proof of Theorem~\ref{t:rank-hd}, we show that $N = \CC^d_n \cap
\UU_{K_{d+1,d+1}}$ is a $\{K_{d+2},K_{d+1,d+1}\}$-matroid different from
$\HH_n^d$.  
\begin{question}
\label{q:extend_hyper}
  Is $\CC^d_n \cap \UU_{K_{d+1,d+1}}$ contained in a $\{K_{d+2},K_{d+1,d+1}\}$-matroid of rank $dn -  \binom{d+1}2$?
\end{question}
In order to motivate the question above, let us remark that both, the
affirmative or the negative answer could have interesting consequences.

If the answer is affirmative, then this could yield a construction of a new
abstract $d$-rigidity family, as discussed, for example, in~\cite{tyomkyn26}.

If the answer is negative, then we get an even stronger counterexample to
Conjecture~\ref{c:hyperconnectivity}: There is no unique maximal
$\{K_{d+2},K_{d+1,d+1}\}$-matroid. Indeed, let us consider a maximal
$\{K_{d+2},K_{d+1,d+1}\}$-matroid $N'$ containing $N$ and a maximal
$\{K_{d+2},K_{d+1,d+1}\}$-matroid $H'$ containing $\HH^d_n$. Then $N'$ and $H'$
differ as the rank of $N'$ is less than $dn -  \binom{d+1}2$ while the rank of
$H'$ is at least the rank of $H$ which is $dn -
\binom{d+1}2$; see~\cite{kalai85-hyperconnectivity}.

In addition, if the answer is negative, there is a hope to save the maximality
part of Conjecture~\ref{c:hyperconnectivity} by adding a restriction on the
rank:
  
\begin{question}[\cite{tyomkyn26pc}]
  Let $d\geq 1$ and $\XX=\{K_{d+2},K_{d+1,d+1}\}$. Is it true that $\HH_n^d$ is
  the unique maximal $\XX$-matroid among $\XX$-matroids of rank $dn -
\binom{d+1}2$?
\end{question}

\paragraph{Analogous questions for the rigidity matroid.} For all the questions
above, there are analogous questions for the rigidity matroid. Given that the
motivation is essentially the same, we only pose the questions. (In the case of
the first question, we have to be a bit careful regarding the list of circuits
we have.)

\begin{question}
  Let $d\geq 1$ and $\XX=\{\cone^k(K_{d+2-k,d+2-k})\colon k
  \in\{0,\dots,d-3\}\} \cup \{K_{d+2}\}$. Is it true that $\RR_n^d$ is the
unique maximal $\XX$-matroid and $\val_\XX$ is its rank function?
\end{question}

\begin{question}
  Let $d \geq 4$. Is $\CC^d_n \cap \UU_{K_{d+2,d+2}}$ contained in a $\{K_{d+2},K_{d+2,d+2}\}$-matroid of rank $dn -  \binom{d+1}2$?
\end{question}

\begin{question}[\cite{tyomkyn26pc}]
  Let $d\geq 3$ and $\XX=\{K_{d+2},K_{d+2,d+2}\}$. Is it true that $\RR_n^d$ is
  the unique maximal $\XX$-matroid among $\XX$-matroids of rank $dn -
\binom{d+1}2$?
\end{question}

\section*{Acknowledgment}
We would like to thank D\'{a}niel Garamv\"{o}lgyi for pointing out to us
a possible problem with Corollary~5.2 in~\cite{kalai85-hyperconnectivity}, Bill
Jackson for a few remarks on the preprint version of the manuscript, and Mykhaylo Tyomkyn for a fruitful discussion regarding open questions.

{\bf Disclaimer:} No AI/LLM models have been used for this research except for
a reference search obtaining~\cite{guterman_nevo25} and~\cite{nevo05-basic}
mentioned below the statement of Proposition~\ref{p:coning_in_H}.

\bibliographystyle{alpha}
\bibliography{main}

\begin{thebibliography}{CJT22b}

\bibitem[AR78]{asimow78-rigidity}
L.~Asimow and B.~Roth.
\newblock The rigidity of graphs.
\newblock {\em Trans. Amer. Math. Soc.}, 245:279--289, 1978.

\bibitem[AR79]{asimow79-rigidity}
L.~Asimow and B.~Roth.
\newblock The rigidity of graphs. {II}.
\newblock {\em J. Math. Anal. Appl.}, 68(1):171--190, 1979.

\bibitem[ASW93]{alfeld93-cofactor}
P.~Alfeld, L.~L. Schumaker, and W.~Whiteley.
\newblock The generic dimension of the space of {$C^1$} splines of degree
  {$d\geq 8$} on tetrahedral decompositions.
\newblock {\em SIAM J. Numer. Anal.}, 30(3):889--920, 1993.

\bibitem[Ber17]{bernstein17-2hyper}
D.~I. Bernstein.
\newblock Completion of tree metrics and rank 2 matrices.
\newblock {\em Linear Algebra Appl.}, 533:1--13, 2017.

\bibitem[CJT22a]{clinch22-rigidity1}
K.~Clinch, B.~Jackson, and S.~Tanigawa.
\newblock Abstract 3-rigidity and bivariate {$C^1_2$}-splines {I}: {W}hiteley's
  maximality conjecture.
\newblock {\em Discrete Anal.}, pages Paper No. 2, 50, 2022.

\bibitem[CJT22b]{clinch22-rigidity2}
K.~Clinch, B.~Jackson, and S.~Tanigawa.
\newblock Abstract 3-rigidity and bivariate {$C^1_2$}-splines {II}:
  {C}ombinatorial characterization.
\newblock {\em Discrete Anal.}, pages Paper No. 3, 32, 2022.

\bibitem[Gar26]{garamvolgyi26pc}
D.~Garamv\"{o}lgyi, 2026.
\newblock Private communication.

\bibitem[GGJ22]{daniel22-rigidity}
D.~Garamv\"olgyi, S.~J. Gortler, and T.~Jord\'an.
\newblock Globally rigid graphs are fully reconstructible.
\newblock {\em Forum Math. Sigma}, 10:Paper No. e51, 27, 2022.

\bibitem[Glu75]{gluck74-rigidity}
H.~Gluck.
\newblock Almost all simply connected closed surfaces are rigid.
\newblock In {\em Geometric topology ({P}roc. {C}onf., {P}ark {C}ity, {U}tah,
  1974)}, Lecture Notes in Math., Vol. 438, pages 225--239. Springer,
  Berlin-New York, 1975.

\bibitem[GN25]{guterman_nevo25}
L.~Guterman and E.~Nevo.
\newblock Matroids arising from algebraic shifting.
\newblock {\em arXiv preprint arXiv:2512.03294}, 2025.

\bibitem[Gra91]{graver91-abstractrig}
J.~E. Graver.
\newblock Rigidity matroids.
\newblock {\em SIAM J. Discrete Math.}, 4(3):355--368, 1991.

\bibitem[JT24]{jackson24-maximal}
B.~Jackson and S.~Tanigawa.
\newblock Maximal matroids in weak order posets.
\newblock {\em J. Combin. Theory Ser. B}, 165:20--46, 2024.

\bibitem[Kal85]{kalai85-hyperconnectivity}
G.~Kalai.
\newblock Hyperconnectivity of graphs.
\newblock {\em Graphs Combin.}, 1(1):65--79, 1985.

\bibitem[Kal90]{kalai90-symmetric}
G.~Kalai.
\newblock Symmetric matroids.
\newblock {\em J. Combin. Theory Ser. B}, 50(1):54--64, 1990.

\bibitem[Max64]{maxwell64-rigidity}
J.~C. Maxwell.
\newblock L{.} on the calculation of the equilibrium and stiffness of frames.
\newblock {\em The London, Edinburgh, and Dublin Philosophical Magazine and
  Journal of Science}, 27(182):294--299, 1864.

\bibitem[Nev05]{nevo05-basic}
E.~Nevo.
\newblock Algebraic shifting and basic constructions on simplicial complexes.
\newblock {\em J. Algebraic Combin.}, 22(4):411--433, 2005.

\bibitem[Oxl11]{oxley11-matroid}
J.~Oxley.
\newblock {\em Matroid theory}, volume~21 of {\em Oxford Graduate Texts in
  Mathematics}.
\newblock Oxford University Press, Oxford, second edition, 2011.

\bibitem[Tyo26a]{tyomkyn26pc}
M.~Tyomkyn, 2026.
\newblock Private communication.

\bibitem[Tyo26b]{tyomkyn26}
M.~Tyomkyn.
\newblock On plane rigidity matroids.
\newblock {\em arXiv preprint arXiv:2602.11892}, 2026.

\bibitem[Whi88]{whiteley88-cofactor}
W.~Whiteley.
\newblock The union of matroids and the rigidity of frameworks.
\newblock {\em SIAM J. Discrete Math.}, 1(2):237--255, 1988.

\bibitem[Whi91]{whiteley91-cofactor}
W.~Whiteley.
\newblock The combinatorics of bivariate splines.
\newblock In {\em Applied geometry and discrete mathematics}, volume~4 of {\em
  DIMACS Ser. Discrete Math. Theoret. Comput. Sci.}, pages 587--608. Amer.
  Math. Soc., Providence, RI, 1991.

\bibitem[Whi96]{whiteley96-geometry}
W.~Whiteley.
\newblock Some matroids from discrete applied geometry.
\newblock In {\em Matroid theory ({S}eattle, {WA}, 1995)}, volume 197 of {\em
  Contemp. Math.}, pages 171--311. Amer. Math. Soc., Providence, RI, 1996.

\end{thebibliography}

\appendix

\section*{Cones in the hyperconnectivity matroid}
In this section we prove Proposition~\ref{p:coning_in_H}(b). However, we need
noticeable preliminaries. We closely follow the notation
in~\cite{kalai85-hyperconnectivity}. We recall the most of the notation from
Section~2 (and partially Section~3) of~\cite{kalai85-hyperconnectivity}; though we refer there for more
details.

\paragraph{The exterior algebra and the left interior product.} 
By $V$ we denote the real vector space with the standard basis $(e_1, \dots,
e_n)$. The exterior algebra $\bigwedge V$ has a basis $(e_S \colon S \subseteq
[n])$
and there is a multilinear product operation $\wedge$ on $\bigwedge V$
satisfying 

\begin{itemize}
  \item[(E1)] $e_\emptyset \wedge e_i = e_i \wedge e_\emptyset = e_i$ for $i
\in [n]$; 

\item[(E2)] $e_S = e_{s_1} \wedge \cdots \wedge e_{s_k}$ for $S = \{s_1, \dots, s_k\}$ with
$s_1 < \dots < s_k$ where $e_i$ is identified with $e_{\{i\}}$ for $i \in [n]$;
and 

\item[(E3)] $e_i \wedge e_j = -e_j \wedge e_i$ for $i,j \in [n]$. In
  particular, $e_i \wedge e_i = 0$. 
\end{itemize}

Then
$\bigwedge^k V$ denotes the subspace spanned by $e_S$ with $|S| = k$.

Given another basis $(f_1, \dots, f_n)$ of $V$, there is a basis $(f_S \colon S
\subseteq [n])$ of $\bigwedge V$ given by $f_S := f_{s_1} \wedge \cdots \wedge f_{s_k}$ for $S =
\{s_1, \dots, s_k\}$ with $s_1 < \dots < s_k$. Then $f_S$ satisfy analogous
condtions to (E1), (E2) and (E3).

Given the dual space $V^*$, then $(e^*_S\colon S \subseteq [n])$ (or
$(f^*_S\colon S \subseteq [n])$) form a basis of $\bigwedge V^*$ (after
identification of $\bigwedge V^*$ with $(\bigwedge V)^*$.
Following~\cite{kalai85-hyperconnectivity}, for $f^* \in \bigwedge V^*$ and $g
\in \bigwedge V$, $f^*(g)$ is denoted by $\langle f^*, g \rangle$. In
particular $\langle e^*_S, e_T \rangle = \delta_{S,T}$ for $S, T \subseteq [n]$
where $\delta_{S,T} = 1$ if $S = T$; otherwise $\delta_{S,T} = 0$. The same
property holds, of course, for $f^*_S$ and $f_T$ as well.

Given $g^* \in \bigwedge^{\ell} V^*$ and $f \in \bigwedge^{k + \ell} V$, there
is an element $g^* \lip f \in \bigwedge^k V$ called the \emph{left
interior product}. It is a multilinear function satisfying the following properties:
\begin{itemize}
 \item[(L1)] $\langle u^*, g^* \lip f\rangle = \langle u^* \wedge g^*,
   f\rangle$ for every $u^* \in \bigwedge V^*$;
 \item[(L2)] $h^* \lip (g^* \lip f) = (h^* \wedge g^*) \lip f$;
 \item[(L3)] $f_T^* \lip f_S = 
   \begin{cases} 
     \pm f_{S \setminus T} & \text{if } T \subseteq S \\
     0 & \text{otherwise} \\
   \end{cases}
     $ for any basis $(f_1, \dots, f_n)$ of $V$ including $(e_1, \dots, e_n)$.
\end{itemize}
The sign in (L3) is determined uniquely, however we do not need it explicitly.
\paragraph{Generic hyperconnectivity matroid.}
From now on, we assume that $(f_1, \dots, f_n)$ is a \emph{generic} basis
of $V$. This means that $\langle f_i, e_j \rangle$ (equivalently $\langle
f^*_i, e_j \rangle$) for $i,j \in [n]$ are algebraically independent over $\Q$.
A \emph{$d$-multicycle} is $m \in \bigwedge^2 V$ such that $f^*_i \lip m = 0$
for $i \in [d]$. The support $\supp(x)$ of $x \in \bigwedge V$ is a collection of those $S
\subseteq [n]$ such that $\alpha_S \neq 0$ when $x$ is written as $\sum_{S
\subseteq [n]} \alpha_S e_S$; equivalently, $\langle e_S, x\rangle \neq 0$. A
graph $G$ with $V(G) \subseteq [n]$ is \emph{dependent} in $\HH_n^d$ if there
is a non-zero multicycle $m$ with $\supp(m) \subseteq G$. (Here we identify a
graph with its edges). This defines a matroid and this matroid is uniquely
determined provided that we started with a generic basis. (For more details we
refer to Section~3 of~\cite{kalai85-hyperconnectivity}.) A graph $G$ is
\emph{$d$-hypercycle}, if there is a $d$-multicycle $x$ with $\supp(x) = G$. If $d$ is
understood from the context, we omit it in the notions multicycle/hypercycle.

\begin{lemma}
  \label{l:contains_hypercycle}
  Let $G,H$ be two distinct hypercycles in $\HH^d_n$ and $e\in G\cap H$. Then, $(H\cup G)\setminus e$ contains a hypercycle in $\HH^{d}_n$. 
\end{lemma}

\begin{proof}
  Since $G$ and $H$ are distinct there are $x,y$ such that $\supp(x)=G$, $\supp(y)=H$ and $x$ is not a scalar multiple of $y$.
  Then there exists $\alpha\in \F\setminus 0$ such that $z = x -\alpha y$
  satisfies $\supp(z)\subseteq (G\cup H) \setminus e$. The support of $z$ is the desired hypercycle.
\end{proof}

\begin{lemma}
  \label{l:union-hypercycle}
  A graph $G$ is a hypercycle in $\HH^d_n$ if and only if $G$ is a union of circuits in $\HH^d_n$.
\end{lemma}
\begin{proof}
  First assume that $G$ is a hypercycle, and we want to show that it is a union
  of circuits. We prove this by induction in the number of edges of $G$. (There
  is nothing to do if this number is $0$.) As $G$ is dependent, it contains a
  circuit $H$.
  Let $x$ and $y$ be multicycles with $\supp(x) = G$ and $\supp(y)
  = H$. Let $e$ be an edge of $G$. Let $z=x-\alpha y$, where $\alpha\neq 0$ is chosen such that
  $G'=\supp(z)\subseteq G\setminus e$. If $G' = \emptyset$, then $G = H$.
  Otherwise, $G'$ is a hypercycle which is by induction union of circuits and
  $G = H \cup G'$.

  If $G$ is a union of circuits, then a generic linear combinations of the
  multicycles corresponding to these circuits is a multicycle $x$ with
  $\supp(x) = G$.
\end{proof}

The following lemma will be needed in the proof of Proposition~\ref{p:cone_hypercycle_hd}. Its notation is adjusted to this
application.

\begin{lemma}
  \label{l:genericsupport}
  Let $W = \SPAN \{e_1, \dots, e_{n-1}\} \leq V$. Let $(g_1, \dots, \widehat
  {g_{d+1}}, \dots, g_n)$ be\footnote{The notation $\widehat{g_{d+1}}$
  indicates that $g_{d+1}$ misses in the list.} a generic basis of $W$ for some $d \in \{0, \dots,
  n-1\}$. Then there is $f_{d+1} \in V$ with the following properties:
\begin{enumerate}[(i)]
  \item The basis $(g_1, \dots, g_d, f_{d+1}, g_{d+2}, \dots, g_n)$ is weakly
    generic in the sense that all (nonzero) coefficients $\langle g_i, e_j
    \rangle$ for $i \in [n] \setminus \{d+1\}, j \in [n-1]$ and $\langle
    f_{d+1}, e_j\rangle$ for $j \in [n]$ are algebraically independent. 
  \item
  Let a graph $G$ be a \hbox{$d$-hypercycle} which does not contain the vertex
    $n$. Then, there exists a multicycle $x$ (with respect to the basis $(g_1, \dots, \widehat {g_{d+1}}, \dots, g_n)$) whose
  support is $G$ such that $f^*_{d+1}\lip x = \sum_{i\in V(G)}e_ib_i$ with
  $b_i\neq 0$ for every $i\in V(G)$.
\end{enumerate}
\end{lemma}

\begin{proof}
  For every such
  (nonempty) hypercycle $G$, fix any multicycle $x = x(G)$ with $\supp(x) = G$.
  By multilinearity of $\lip$, the function $\psi_{x,i}$ assigning to $f_{d+1} \in V$ the
  value $b_i \in \R$ in the decomposition $f^*_{d+1}\lip x = \sum_{i\in G}e_ib_i$ is
  linear for every fixed $i$ and $x$. Assume that $x = \sum_{\{j,k\} \in G}
  \alpha_{\{j,k\}} e_{\{j,k\}}$. Then (L3),
  used with the basis $(e_1, \dots, e_n)$, gives that $\psi_{x,i} = 0$ if $i
  \not \in V(G)$ while it is a nonzero map for $i \in V(G)$. (We recall that we
  identify a graph with its edges, then $G$ has no isolated vertices.)
  Thus, for $i \in
  V(G)$ those $f
  _{d+1}$ in $V$ for which $\psi_{x,i}(f_{d+1}) = 0$ form a hyperplane in~$V$.
  Altogether, because there are only finitely many $x$ and $i$ to consider, we
  can choose $f_{d+1}$ weakly generically with respect to $(g_1, \dots,
  \widehat{g_{d+1}}, \dots, g_n)$ so that
  all the conditions $\psi_{x,i}(f_{d+1}) = 0$ for $i \in V(G)$ are avoided.
\end{proof}

The following proposition is a generalization of Theorem~5.1
in~\cite{kalai85-hyperconnectivity} (or, equivalently,
Proposition~\ref{p:coning_in_H}(a)---a short proof how to deduce
Proposition~\ref{p:coning_in_H}(a) is given below the proof of the following
proposition.)
%Indeed, Theorem~5.1
%in~\cite{kalai85-hyperconnectivity} is equivalent with the claim that a graph
%$G$ contains a hypercycle in $\HH^d_n$ if and only if $\cone(G)$ contains a
%hypercycle in $\HH^{d+1}_n$. 
This generalization is crucial for our proof of
Proposition~\ref{p:coning_in_H}(b). 

\begin{proposition}
  \label{p:cone_hypercycle_hd}
  Let $G$ be a graph and $n \geq |V(G)|+1$. 
\begin{enumerate}[(a)]
  \item If $G$ is a hypercycle in $\HH^d_n$ then $\cone(G)$ is a hypercycle in
    $\HH^{d+1}_n$.
  \item  If $\cone(G)$ contains a hypercycle $H$ in $\HH^{d+1}_n$ such that $G$
    is a subgraph of $H$, then $G$ is a hypercycle in $\HH^d_n$. In particular,
    if $\cone(G)$ is a hypercycle in $\HH^{d+1}_n$, then $G$ is a hypercycle in
    $\HH^d_n$.
\end{enumerate}
\end{proposition}

\begin{proof}
  Without loss of generality, let $n$ be the apex of $\cone(G)$.

  Let $W = \SPAN\{e_1, \dots, e_{n-1}\} \leq V$ and
  $\gvec=(g_1,\dots,\widehat{g_{d+1}},\dots, g_n)$ be a generic basis of $W$.
  Let $f_{d+1} \in V$ be as in Lemma~\ref{l:genericsupport}. For $i \in [n]
  \setminus \{d+1\}, j \in [n-1]$ we set $a_{i,j} = \langle g_i, e_j
    \rangle$ and for $j \in [n]$ we set $a_{d+1,j} = \langle f_{d+1}, e_j
    \rangle$.  For $i \in [n]
  \setminus \{d+1\}$, we pick real numbers $a_{i,n}$ so that all $a_{i,j}$ for
  $i,j \in [n]$ are algebraically independent. This is possible due to
  Lemma~\ref{l:genericsupport}(i).

  Now we set 
   \begin{equation}
   \label{e:f_via_g}
      f_i = g_i + \frac{a_{i,n}}{a_{d+1,n}}f_{d+1}
   \end{equation}
    for $i \in [n] \setminus \{d+1\}$. We claim that $\fvec = (f_1,\dots, f_n)$
    is a generic basis of $V$. Indeed, consider $b_{i,j} := \langle f_i,
    e_j\rangle$ for $i,j \in [n]$. Then $a_{i,j}$ can be expressed via
    $b_{i,j}$ as follows: $a_{d+1,j} = b_{d+1,j}$ for $j \in [n]$ by
    definition; $a_{i,n} = b_{i,n}$ for $i \in [n] \setminus \{d+1\}$ by
    using~\eqref{e:f_via_g} for $\langle f_i, e_n\rangle$, and $a_{i,j} =
    b_{i,j} - \frac{b_{i,n} b_{d+1,j}}{b_{d+1,n}}$ for $i \in [n] \setminus
    \{d+1\}, j \in [n-1]$ by using~\eqref{e:f_via_g} for $\langle f_i,
    e_j\rangle$. Thus algebraic independence of $a_{i,j}$ implies algebraic
    independence of $b_{i,j}$ (by comparing the transcendence degrees of
    $\Q(a_{i,j}\colon i,j \in [n])$ and $\Q(b_{i,j}\colon i,j \in [n])$). 

  By Lemma~\ref{l:union-hypercycle} the property of being a hypercycle does not depend
  on the choice of generic basis, and $n-1\geq |V(G)|$, thus we may take $\fvec$ as
  a basis in the definition of $\HH_n^{d+1}$ and $\gvec$ as a basis in the
  definition of $\HH_{n-1}^d$. 

  First we prove (a). We assume that $G$ is a hypercycle in $\HH_n^{d}$ so it is a hypercycle in
  $\HH_{n-1}^d$ as well. Let $x=\sum_{\{i,j\} \in G}\alpha_{\{i,j\}} e_{\{i,j\}}$ be a
  $d$-multicycle (with respect to the basis $\gvec$) with $\supp(x) = G$ satisfying the conclusion of
  Lemma~\ref{l:genericsupport}(ii). We set $y = f^*_{d+1}\lip (x\wedge e_n)$
  and our aim is to show that $y$ is a $(d+1)$-multicycle with $\supp(y) =
  \cone(G)$.
  First, by (L2) and $f^*_{d+1} \wedge f^*_{d+1} = 0$, we get $f^*_{d+1}\lip y
  = 0$. For $i\in [d]$ we have that
  \[ f^*_i\lip y = (g^*_i+\frac{\langle f_i,e_n\rangle}{\langle
  f_{d+1},e_n\rangle} f^*_{d+1})\lip y = g^*_i \lip y = 
  -f^*_{d+1} \lip(g^*_i\lip (x\wedge e_n)),
  \]
  using $f^*_{d+1}\lip y  = 0$, (L2) and (E3).

 Now we claim that $g^*_i\lip (x\wedge e_n) = 0$. Note that $g^*_i\lip (x\wedge
  e_n) \in \bigwedge^2 V$, thus it is sufficient to check that $\langle
  e^*_{\{j,k\}}, g^*_i\lip (x\wedge  e_n) \rangle = 0$ for every $j,k \in [n],
  j < k$. By (L1), we need to check that $\langle e^*_{\{j,k\}} \wedge g^*_i, x
  \wedge e_n \rangle = 0$. If $k \neq n$, then $\supp(e^*_{\{j,k\}} \wedge
  g^*_i)$ does not contain any triple containing $n$ as we set up $g_i$ so
  that $\langle g_i, e_n \rangle = 0$. On the other hand, every triple in
  $\supp(x \wedge e_n)$ contains $n$ because $n$ is not a vertex of $G$ and
  $\supp(x) = G$. Thus $\langle e^*_{\{j,k\}} \wedge g^*_i, x
  \wedge e_n \rangle = 0$ as required. If $k = n$, then $\langle e^*_{\{j,n\}}
  \wedge g^*_i, x  \wedge e_n \rangle = 0$ is equivalent with $\langle e^*_j
  \wedge g^*_i, x \rangle = 0$. By (L1), this is equivalent with $\langle
  e^*_j, g^*_i \lip x \rangle = 0$ which follows from the fact that $x$ is a
  $d$-multicycle. Altogether we have verified that $y$ is a $(d+1)$-multicycle.

  Now we verify that $\supp(y) = \cone(G)$. It is sufficient to check $\langle
  e^*_{\{i,j\}}, y \rangle \neq 0$ for $\{i,j\} \in \cone(G)$, $i < j$. By
  (L1), we have 
  \begin{equation}
  \label{e:ey}
    \langle e^*_{\{i,j\}}, y \rangle = \langle e^*_{\{i,j\}} \wedge f^*_{d+1},
    x \wedge e_n \rangle.
  \end{equation}

  First
  assume $\{i,j\} \in G$, that is, $j < n$. Then $\supp(e^*_{\{i,j\}} \wedge
  f^*_{d+1})$ consist of triples $\{i,j,k\}$, where $k \in [n] \setminus
  \{i,j\}$ by genericity of $f^*_{d+1}$ while $\supp(x \wedge e_n) =
  \supp(\sum_{\{i',j'\} \in G} \alpha_{\{i',j'\}} e_{\{i',j',n\}})$ consists of
  triples $\{i',j',n\}$ where $\{i',j'\} \in G$. Those two supports intersect
  exactly in the triple $\{i,j,n\}$ thus the value in~\eqref{e:ey} is nonzero.

  Now assume $\{i,j\} \in \cone(G) \setminus G$, that is, $j=n$. Then the
  expression in~\eqref{e:ey} is nonzero if and only if $\langle e_i^* \wedge
  f^*_{d+1}, x \rangle$ is nonzero. By (L1) this is equal to $\langle e_i^*,
  f^*_{d+1} \lip x \rangle$. However, this is exactly $b_i$ from
  Lemma~\ref{l:genericsupport}(ii), thus nonzero. This finishes the proof that
  $\cone(G)$ is a hypercycle in $\HH^{d+1}_n$.

  Now we prove (b). We assume that $\cone(G)$ contains a hypercycle $H$ in
  $\HH^{d+1}_n$ with $G \subseteq H$ and we aim to show
  that $G$ is a hypercycle in $\HH^d_n$. Let $y$ be a $(d+1)$-multicycle
  with $\supp(y) = H$ and let us write it as $y = x + \sum_{i\in
  V(G)}\alpha_i e_{\{i,n\}}$ where $\supp(x) = G$. As we assume $f^*_i \lip y =
  0$ for $i \in [d+1]$, we also get $g^*_i \lip y = 0$ for $i \in [d]$
  from~\eqref{e:f_via_g}. Using the fact that $\langle g_i, e_n\rangle = 0$ and
  using (L3) with the base $(e_1, \dots, e_n)$, we get that $g^*_i \lip (\sum_{i\in
  V(G)}\alpha_i e_{\{i,n\}}) = ce_n$ for some $c$ real. From $\langle g^*_i
  \lip y, e_n\rangle = 0$ and $\langle g^*_i\lip x, e_n\rangle = 0$ (this
  again follows from (L3) using that $\langle g^*_i, e_n \rangle = 0$ and
  $\supp(x) = G$) we deduce that $c=0$. Thus $g^*_i \lip y = 0$
  implies $g^*_i \lip x = 0$. That is, $x$ is a $d$-multicycle with $\supp(x) =
  G$.
\end{proof}

\begin{proof}[Proof of Proposition~\ref{p:coning_in_H}(a)]
First assume that $G$ is dependent in $\HH_n^d$. This means that $G$ contains a
  hypercycle $G'$ in $\HH_n^d$. By Proposition~\ref{p:cone_hypercycle_hd}(a),
  $\cone(G')$ is a hypercycle in $\HH_n^{d+1}$. Thus $\cone(G)$ is dependent in
  $\HH_n^{d+1}$.

  Now assume that $\cone(G)$ is dependent in $\HH_n^{d+1}$, thus it contains a
  hypercycle $H$. Let $G'$ be the intersection of $G$ and $H$. (Note that $G'$
  is nonempty as any star graph is independent in $\HH_n^{d+1}$.) By
  Proposition~\ref{p:cone_hypercycle_hd}(b) (used with $G'$ instead of $G$), we
  get that $G'$ is a hypercycle in $\HH_n^d$. Thus $G$ is dependent in
  $\HH_n^d$.
\end{proof}

\begin{proof}[Proof of Proposition~\ref{p:coning_in_H}(b)]
  First, assume that $\cone(G)$ is a circuit, while $G$ is not.
  By Proposition~\ref{p:cone_hypercycle_hd}(b)
%  By~\cite[Theorem 5.1]{kalai85-hyperconnectivity} 
  the graph $G$ is dependent in $\HH^d_n$. Then
  it contains a proper subgraph $H\subseteq G$ that is a circuit.
%  By~\cite[Theorem 5.1]{kalai85-hyperconnectivity} 
  By Proposition~\ref{p:cone_hypercycle_hd}(a), the proper
  subgraph $\cone(H)$ of $\cone(G)$ is dependent, a contradiction.
  %By the reverse implication of Lemma~\ref{l:cone-hd}, $\cone(H)$ is a
  %proper subgraph of $\cone(G)$, in particular dependent, a contradiction.
  
  Now, assume that $G$ is a circuit, while $\cone(G)$ is not.
  By Proposition~\ref{p:cone_hypercycle_hd}(a) the graph $\cone(G)$ is
  a hypercycle in $\HH^{d+1}_n$. Then it contains a proper subgraph
  $H\subseteq \cone(G)$ that is a circuit. In particular, $H$ is a hypercycle distinct
  from $\cone(G)$.
  Then, either $H$ is contained in a cone of a proper subgraph of $G$, or not.  
  If the former is true, i.e., $H \subseteq \cone(G')$ with $G'\subseteq G$,
  Proposition~\ref{p:cone_hypercycle_hd}(b) implies that $G'$ is a dependent
  proper subgraph of $G$, a contradiction.
  On the other hand, if $G\subseteq H$, let $e\in G$. By
  Lemma~\ref{l:contains_hypercycle}
  we can conclude that $(\cone(G)\cup H) \setminus e = \cone(G\setminus e)$ is dependent, a contradiction by the first case.
\end{proof}
\end{document}